\documentclass[a4paper,10pt]{amsart}
\usepackage[utf8]{inputenc}
\usepackage[T1]{fontenc}
\usepackage[hmargin={3.5cm, 3.0cm}, vmargin={3.5cm, 3.5cm}]{geometry}
\usepackage[danish,english]{babel}
\usepackage{amsmath,amssymb, amsthm, amsfonts}
\usepackage[all,cmtip]{xy}
\usepackage{enumerate}
\usepackage{mathdots}
\theoremstyle{plain} 
		\newtheorem{theo}{Theorem}[section] 
		\newtheorem{lem}[theo]{Lemma}
		\newtheorem{pro}[theo]{Proposition}
		\newtheorem{Corr}[theo]{Corollary}
		\newtheorem{observation}[theo]{\textbf{Observation}}
\theoremstyle{definition} 
		\newtheorem{defi}[theo]{Definition}
		
\theoremstyle{remark}

\theoremstyle{plain} 
		
\theoremstyle{definition} 
		
\theoremstyle{remark}

\newcommand{\N}{{\mathbb N}}

\newcommand{\R}{{\mathbb R}}
\newcommand{\C}{{\mathbb C}}

\newcommand{\8}{{\infty}}
\newcommand{\ra}{{\rightarrow}}
\newcommand{\ip}[2]{\langle {#1} , {#2} \rangle}

\newcommand{\actson}{\curvearrowright}
\newcommand{\abs}[1]{\left\lvert #1 \right\rvert}
\newcommand{\norm}[1]{\left\lVert #1 \right\rVert}
\DeclareMathOperator{\Tube}{Tube}

\renewcommand\phi\varphi
\renewcommand\epsilon\varepsilon
\newcommand{\claimmark}{\hfill$\heartsuit$}
\DeclareMathOperator{\supp}{supp}
\usepackage{xspace}
\def\lcsc{l.c.s.c.\@\xspace}

\newbox\hangbox
\def\hang#1{\setbox\hangbox=\hbox{#1}\hspace*{-\wd\hangbox}\box\hangbox}

\begin{document}
\selectlanguage{english}

\title{Property A and uniform embedding for locally compact groups}
\author{Steven Deprez$^{1}$ and Kang Li$^{2}$}
\thanks{\hang{$^{1}$ }corresponding author, University of Copenhagen, sdeprez@math.ku.dk}
\thanks{Supported by ERC Advanced Grant no. OAFPG 247321}
\thanks{\hang{$^{2}$ }University of Copenhagen, kang.li@math.ku.dk}
\thanks{\hang{$^{1,2}$ }Supported by the Danish National Research Foundation through the Centre for Symmetry and Deformation}
\thanks{(DNRF92).} % bad fix
%\date{\today}
\begin{abstract}
For locally compact groups, we define an analogue to Yu's property A that he defined for discrete metric spaces. We show that our property A for locally compact groups agrees with
Roe's notion of property A for proper metric spaces, defined in \cite{R05}.
We prove that many of the results that are known to hold in the discrete setting, hold also in the locally compact setting. In particular, we show that property A is equivalent to amenability at
infinity (see \cite{HR00} for the discrete case), and that a locally compact group with property A embeds uniformly into a Hilbert space (see \cite{Yu00} for the discrete case).
We also prove that the Baum-Connes assembly map with coefficients
is split-injective, for every locally compact group that
embeds uniformly into a Hilbert space. This extends results by Skandalis, Tu and Yu \cite{STY02}, and by Chabert, Echterhoff and Oyono-Oyono \cite{CEO04}.
\end{abstract}
\maketitle


%%%%%%%%%% Begin Abstract %%%%%%%%%%
%%%%%%%%%%% End Abstract %%%%%%%%%%%

%%%%%%%%%% Begin Acknowledgments %%%%%%%%%%

%%%%%%%%%%% End Acknowledgments %%%%%%%%%%%

%%%%%%%%%% Begin Table of contents %%%%%%%%%%

%%%%%%%%%%% End Table of contents %%%%%%%%%%%

%%%%%%%%%% Begin Text %%%%%%%%%%
\section{Introduction}
%For locally compact groups, there are many different generalizations of the concept of exact discrete groups. We say that a locally compact group $G$ is exact if its reduced group C$^\ast$-algebra
%C$^\ast_r(G)$ is an exact C$^\ast$-algebra. A stronger condition is that the reduced crossed product with $G$ is an exact functor from the category of separable C$^\ast$-algebras to itself. A group that
%satisfies this condition is said to be Kirchberg-Wasserman-exact, or KW-exact for short. A locally compact group $G$ is said to be amenable at infinity if $G$ has an amenable action on a compact space $Y$.
%The action $G\actson Y$ is amenable if and only if its transformation groupoid $Y\rtimes G$ is amenable.
%Yu defined Property A for discrete metric spaces in \cite{...Yu}. We say that a locally compact group
%$G$ has property A if it has property A as a metric space, for any (or equivalently all) proper metric on $G$.
%
%In the discrete case, all these three notions are well-known to be equivalent, see for example \cite{Brown-Ozawa?}, but in the locally compact case, this is not known. However, it is known that amenability
%at infinity implies KW-exactness \cite{???}, which clearly implies plain exactness. It is an important
%open question if any two of these conditions are equivalent.
%
Gromov introduced the notion of uniform embeddability of metric spaces and suggested that finitely generated discrete groups that are uniformly embeddable in a Hilbert space, when viewed as metric spaces
with a word length metric, might satisfy the Novikov conjecture \cite{FRR95,G93}.
Yu showed that this is indeed the case, provided that the classifying space is a finite CW-complex \cite{Yu00}. In the same paper Yu introduced a weak form of amenability on discrete metric spaces that he
called property A, which guarantees the existence of a uniform embedding into Hilbert space. Higson and Roe observed in \cite{HR00} that the metric space underlying a finitely generated discrete group has
property A if and only if it admits a topologically amenable action on some compact Hausdorff space. Ozawa showed in \cite{Oza00} that a discrete group admits a topologically amenable action on a compact Hausdorff space if and only if the group is exact.
In the case of property A groups, Higson strengthened Yu's result by removing the finiteness
assumption on the classifying space \cite{H00}.
Indeed, he proved that the Baum-Connes assembly map with coefficients, for any countable discrete group
which has a topologically amenable action on a compact Hausdorff space,
is split-injective.
Baum, Connes and Higson showed that this implies the Novikov conjecture \cite{BCH94}.
Using Higson's descent technique (see \cite{H00}), Skandalis, Tu and Yu \cite{STY02} were able to generalize the split-injectivity result
to arbitrary discrete groups which admit a uniform embedding into Hilbert space, and hence they answered Gromov's question.
% Later on, Tu \cite{Tu04} extended this result even further by showing that a
%discrete
%group which admits a uniform embedding into Hilbert space has a $\gamma$ element. Whenever a locally compact group has a $\gamma$ element, then its Baum-Connes assembly map with coefficients is
%split-injective \cite{TuHyp99}.

In \cite{R05}, Roe generalized property A to proper metric spaces with bounded geometry (in the sense of \cite{Roe95}). All second countable locally compact groups have a proper left-invariant metric that implements the
topology, and such a metric is unique up to coarse equivalence (see \cite{HP06} and \cite{S74}).
Moreover, locally compact groups with a proper left-invariant metric, have bounded geometry (see \cite{HP06}). Roe already proved that his generalization of
property A is equivalent to Ozawa's notion of exactness, see \cite{R05}. A locally compact group $G$ satisfies Ozawa's notion of exactness if there
is a net of positive type kernels
$k_i:G\times G\rightarrow\C$ that tends to $1$ uniformly on tubes, and such that each $k_i$ is supported in a tube. We say that a subset $T\subseteq G\times G$ is a tube if it is contained in a set of the
form $\Tube(K)=\{(s,t): s^{-1}t\in K\}\subseteq G\times G$ for some compact subset $K\subseteq G$.
Anantharaman-Delaroche has shown in \cite{CA02} that whenever a locally compact group admits a topologically
amenable action on a compact Hausdorff space, it also satisfies Ozawa's notion of exactness. 

We give an alternative definition of property A, that resembles more closely Yu's definition,
and we show that it is equivalent to Roe's definition. Moreover, we give a direct and elementary proof that it is equivalent to Ozawa's notion of exactness.
We continue by showing that a locally compact group has property A if and only if it has a topologically amenable action on a compact Hausdorff space. This statement was proven for discrete groups
by Higson and Roe \cite{HR00}.

Whenever a locally compact group admits a topologically amenable action on a compact Hausdorff space, it is uniformly embeddable into a Hilbert space (see \cite{CA02}). By the above, this is also true for groups with property A.
We also give an alternative characterisation of the locally compact groups that embed uniformly into a Hilbert space. We say that an action $G\actson X$ on a compact Hausdorff space $X$ has the Haagerup property
if its transformation
groupoid $X\rtimes G$ admits a continuous proper conditionally negative type function. Then we show that a locally compact group embeds uniformly into a Hilbert space
if and only if it admits a Haagerup action on a compact Hausdorff space.
Finally, we apply Higson's descent technique and the going--down functor of Chabert, Echterhoff and Oyono-Oyono, to obtain an analogue of the result of Skandalis, Tu and Yu (see \cite{STY02}):
we show that the Baum-Connes assembly map is split-injective for all locally
compact groups that embed uniformly into Hilbert space.

\subsection*{Acknowledgements}
The authors wish to thank Uffe Haagerup and Ryszard Nest for suggesting the topic of this paper and for valuable discussions. We would also like to thank Claire Anantharaman-Delaroche for inspiring conversations.
\section{Property A for locally compact groups}
In this section, we introduce our own notion of property A for locally compact second countable (\lcsc) groups.
Yu first introduced property A for discrete metric spaces in \cite{Yu00}. Our own notion of property A is closely modelled on Yu's definition.
Roe has introduced a generalization of property A for proper metric spaces with bounded geometry (see \cite{R05}).
Every second countable locally compact group $G$ has a proper left-invariant metric $d$ that implements the topology on $G$. This metric is unique up to coarse equivalence. Moreover,
the proper metric space $(G,d)$ has bounded geometry, see \cite{HP06} and \cite{S74}.
So Roe's property A makes sense for \lcsc groups, and we show in theorem \ref{equA} that
it agrees with our property A. We combine theorem \ref{equA} with Anantharaman-Delaroche's description of groups that are amenable at infinity \cite[Proposition 3.4 and 3.5]{CA02},
i.e. groups that admit a topologically amenable action on a compact Hausdorff space.
In this way we obtain corollary \ref{AAmen}: a \lcsc group has property A if and only if it is amenable at
infinity.

For the rest of the paper, $G$ will always denote a \lcsc group. We fix a left Haar measure $\mu$ on $G$.
We also consider the measure $\mu^\prime$ on $G\times\N$ which is the product measure of $\mu$ with the counting measure on $\N$.

%\begin{defi}
%Let $G$ be a locally compact, second countable, Hausdorff topological group. $G$ has property A if for any $R, \epsilon>0$, there exist $S>0$ and a family $\{A_s\}_{s\in G}$ of Borel subsets of
% $G\times \N$ with $0<\mu'(A_s)<\8$ such that
%\begin{itemize}
%\item{for all $s,t\in G$ with $d(s,t)\leq R$ we have
%\begin{align*}
%\frac{\mu'(A_s\Delta A_t)}{\mu'(A_s\cap A_t)}<\epsilon,
%\end{align*}}
%\item{$(t,n)\in A_s$ implies $d(s,t)\leq S$.}
%\end{itemize}
%\end{defi}
%
%Property A is a coarse property of the metric space $(G,d)$. In particular, property A does not depend on the choice of the proper left-invariant metric $d$.
%For later use, we reformulate property A in purely topological terms. To do this, we need some terminology.

\begin{defi}
  Let $G$ be a \lcsc group and let $K\subseteq G$ be a compact subset. Then we write
  \[\Tube(K)=\{(s,t)\in G\times G : s^{-1}t\in K\}.\]
  We say that a subset $T\subseteq G\times G$ is a \emph{tube} if $\{s^{-1}t:(s,t)\in T\}$ is precompact,
  or equivalently, if $T\subseteq \Tube(K)$ for some compact subset $K\subseteq G$.
\end{defi}

\begin{defi} A l.c.s.c group $G$ has property A if
for any compact subset $K\subseteq G$ and $\epsilon >0$, there exist a compact subset $L\subseteq G$ and a family $\{A_s\}_{s\in G}$ of Borel subsets of $G\times \N$ with $0<\mu'(A_s)<\8$ such that
\begin{itemize}
\item for all $(s,t)\in \Tube(K)$ we have
  \begin{align*}
    \frac{\mu'(A_s\Delta A_t)}{\mu'(A_s\cap A_t)}<\epsilon,
  \end{align*}
\item$(t,n)\in A_s$ implies $(s,t)\in \Tube(L)$.
\end{itemize}
\end{defi}

Theorem \ref{equA} below gives a number of equivalent characterizations of property A. It is an extention of \cite{Tu01} to the locally compact case. Condition (3) says that $G$ has property A in the sense
of Roe \cite{R05}, as a proper metric space with bounded geometry.
Condition (5) is Ozawa's notion of exactness for \lcsc groups \cite{Oza00}. The equivalence between (3) and (5) has already been proven in \cite{R05},
by reducing the problem to the discrete case. Nevertheless, we provide a simple direct proof of their equivalence, for the sake of completeness.

\begin{theo}\label{equA}
Let $G$ be a \lcsc group. The following are equivalent:
\begin{itemize}
\item[$1)$] $G$ has property A;\\
\item[$2)$] For any compact subset $K\subseteq G$ and $\epsilon>0$, there exist a compact subset $L\subseteq G$ and a continuous map $\eta:G\ra L^1(G)$ such that $||\eta_t||_1=1$, $\supp\eta_t\subseteq tL$ for every $t\in G$ and
\begin{align*}
\sup_{(s,t)\in \Tube(K)}||\eta_s-\eta_t||_1<\epsilon;
\end{align*}
\item[$3)$] For any compact subset $K\subseteq G$ and $\epsilon>0$, there exist a compact subset $L\subseteq G$ and a weak-$*$ continuous map $\nu:G\ra C_0(G)^*_+$ such that $||\nu_t||=1$, $\supp\nu_t\subseteq tL$ for every $t\in G$ and
\begin{align*}
\sup_{(s,t)\in \Tube(K)}||\nu_s-\nu_t||<\epsilon;
\end{align*}
\item[$4)$] For any compact subset $K\subseteq G$ and $\epsilon>0$, there exist a compact subset $L\subseteq G$ and a continuous map $\xi:G\ra L^2(G)$ such that $||\xi_t||_2=1$, $\supp\xi_t\subseteq tL$ for every $t\in G$ and
\begin{align*}
\sup_{(s,t)\in \Tube(K)}||\xi_s-\xi_t||_2<\epsilon;
\end{align*}
\item[$5)$] For any compact subset $K\subseteq G$ and $\epsilon>0$, there exist a compact subset $L\subseteq G$ and a continuous positive type kernel $k:G\times G\ra \C$ such that $\supp k\subseteq \Tube(L)$ and
\begin{align*}
\sup_{(s,t)\in \Tube(K)}|k(s,t)-1|<\epsilon.
\end{align*}
\end{itemize}
\end{theo}

In statements (2)-(5), we assume that the maps $\eta,\nu,\xi,k$ are continuous, because this is standard for locally compact groups. But in fact, each of the statements (2)-(5) is equivalent to the
corresponding statement without the continuity assumption. In lemma \ref{lem:cont}, we carefully state and prove this for statement (2). We omit the argument for statements (3)-(5), as it is entirely analogous.
Part of the proof of lemma \ref{lem:cont} below consists of convolving $\eta$ with a ``nice'' function. We use the same class of ``nice'' functions several times in the paper, so we give them a name.

\begin{defi}
A cut-off function for $G$ is a function $f$ in $C_c(G)$ such that
\begin{itemize}
\item{$f\geq 0$;}
\item{$f(t^{-1})=f(t)$ for all $t\in G$;}
\item{$\supp f$ is a compact neighborhood of the unit element $e$ of $G$;}
\item{$\int_G f(t)d\mu(t)=1$.}
\end{itemize}
\end{defi}

Observe that every \lcsc group has cut-off functions.

\begin{lem}
  \label{lem:cont}
  Let $G$ be a \lcsc group. Suppose that $G$ satisfies the following property.
  \begin{itemize}
  \item[$2^\prime)$]
    For any compact subset $K\subseteq G$ and $\epsilon>0$, there exist a compact subset $L\subseteq G$ and a map $\eta:G\ra L^1(G)$ such that $\|\eta_t\|_1=1$,
    $\supp\eta_t\subseteq tL$ for every $t\in G$ and
    \begin{align*}
      \sup_{(s,t)\in \Tube(K)}\|\eta_s-\eta_t\|_1<\epsilon;
    \end{align*}
  \end{itemize}
  Then we can assume that the map $t\mapsto\eta_t$ is continuous, i.e., $G$ satisfies property (2) from theorem \ref{equA} above.
\end{lem}
\begin{proof}
  The proof proceeds in two steps. In step one, we show that we can assume that $t\mapsto\eta_t$ is a Borel map.
  In step two we use a convolution argument to make $t\mapsto\eta_t$ continuous.

  \textbf{Step 1:} We can assume that $t\mapsto\eta_t$ is a piecewise constant Borel map.\\
  Fix any compact subset $C\subseteq G$ with non-empty interior. Since $G$ is second countable, we find a sequence $(s_n)$ in $G$ such that $G=\cup_n s_n C$.
  Define a sequence $(C_n)$ of Borel subsets of $G$ by induction as follows. We set $C_1=s_1C$ and for each $n>1$, we set $C_n=s_nC\setminus(C_1\cup\ldots\cup C_{n-1})$.

  Let $\varepsilon>0$ and let $K\subseteq G$ be a compact subset. Then we see that the product $CKC^{-1}\subseteq G$ is still a compact subset. Since $G$ satisfies our condition $(2^\prime)$,
  we find a map $\eta:G\rightarrow L^1(G)$ and a compact subset $L\subseteq G$ such that $\|\eta_t\|_1=1$,
  $\supp\eta_t\subseteq tL$ for every $t\in G$ and
  \begin{align*}
    \sup_{(s,t)\in \Tube(CKC^{-1})}\|\eta_s-\eta_t\|_1<\epsilon;
  \end{align*}

  Define a Borel map $\xi:G\rightarrow L^1(G)$ setting $\xi_t=\eta_{s_n}$ whenever $t\in C_n$. This is well-defined because $G$ is the disjoint union of the Borel sets $C_n$.
  We see that $\|\xi_t\|_1=1$ for all $t\in G$. Observe that, if $t\in C_n$, then it follows that $s_n\in tC^{-1}$.
  Let $(s,t)\in\Tube(K)$ and take $n,m\in\N$ such that $s\in C_n$, $t\in C_m$. Then we see that $s_n^{-1}s_m\in CKC^{-1}$, so
  \[\|\xi_s-\xi_t\|_1=\|\eta_{s_n}-\eta_{s_m}\|_1<\varepsilon.\]
  Moreover, we compute that $\supp(\xi_t)=\supp(\eta_{s_m})\subseteq s_mL\subseteq tC^{-1}L$.
  It is now clear that $\xi$ is a Borel map that satisfies condition $(2^\prime)$.

  \textbf{Step 2:} We can assume that $t\mapsto\eta_t$ is continuous.\\
  Fix a cut-off function $f:G\rightarrow [0,\infty)$. Denote $C=\supp(f)$. Let $\varepsilon>0$ and let $K\subseteq G$ be a compact subset. By step 1,
  we find a piecewise constant Borel map $\eta:G\rightarrow L^1(G)$ and a compact subset $L\subseteq G$ such that $\|\eta_t\|_1=1$,
  $\supp\eta_t\subseteq tL$ for every $t\in G$ and
  \begin{align*}
    \sup_{(s,t)\in \Tube(CKC^{-1})}\|\eta_s-\eta_t\|_1<\epsilon;
  \end{align*}

  We define $\xi:G\rightarrow L^1(G)$ by the formula
  \[\xi_t(v)=\int_G f(s^{-1}t)\abs{\eta_{s}(v)} d\mu(s).\]
  We check that $\xi_t$ is indeed in $L^1(G)$ and has norm $1$ and still satisfies condition (2$^\prime$).
  Then we show that $\xi$ is a continuous map, and hence staisfies condition (2) of theorem \ref{equA}.

  We compute that
  \[\norm{\xi_t}_1=\int_G\int_G f(t^{-1}s)\abs{\eta_s(v)} d\mu(s)d\mu(v)=\int_G f(t^{-1}s)\norm{\eta_s}_1d\mu(s)=\int_G f(s)d\mu(s)=1.\]
  It is clear that the support of $\xi_t$ is contained in the compact subset $tCL$. Whenever we have $(s,t)\in \Tube(K)$, we see that
  \begin{align*}
    \norm{\xi_s-\xi_t}_1&=\int_G \abs{\int_G f(r^{-1}s)\abs{\eta_r(v)}d\mu(r) - \int_G f(r^{-1}t)\abs{\eta_r(v)}d\mu(r)}d\mu(v)\\
    &\leq\int_G \int_G f(r^{-1})\abs{\eta_{sr}(v)-\eta_{tr}(v)}  d\mu(r)d\mu(v)\\
    &=\int_G f(r)\norm{\eta_{sr}-\eta_{tr}}_1 d\mu(r)\\
    &<\varepsilon \int_G f(r) d\mu(r)=\varepsilon.
  \end{align*}
  where the last inequality follows because $(sr,tr)\in\Tube(CKC^{-1})$ whenever $r^{-1}\in C$.

  Suppose that $(t_n)$ is a sequence in $G$ that tends to $t$. Without loss of generality, we can assume that
  $t_n$ remains in the compact neighborhood $tC$. It follows that
  \begin{align*}
    \left\lVert\xi_{t_n}-\xi_t\right\rVert_1
    &\leq\int_G\int_G \left\lvert f(s^{-1}t_n)-f(s^{-1}t)\right\rvert \left\lvert \eta_s(v)\right\rvert d\mu(s)d\mu(v)\\
    &=\int_G \abs{f(s^{-1}t_n)-f(s^{-1}t)}\norm{\eta_s}_1 d\mu(s).
  \end{align*}
  This last integral converges to 0 by the Lebesgue dominated convergence theorem.

  We have shown that our continuous map $\xi:G\rightarrow L^1(G)$ satisfies the required conditions.
\end{proof}

Throughout this paper, we often need to ``smoothen'' a given kernel. For example,
when we are given a measurable kernel $k_0$ on $G$, we can obtain a continuous kernel by convolving
$k_0$ with a cut-off function. Lemma \ref{lem:kernel} below shows that this convolution procedure preserves a number of relevant properties
of the kernel.
\begin{lem}
  \label{lem:kernel}
  Let $G$ be a l.c.s.c group and let $k_0:G\times G\rightarrow \C$ be a measurable kernel that is bounded on every tube. Let $f:G\rightarrow [0,\8)$
  be a cut-off function for $G$. Define a new kernel $k:G\times G\rightarrow \C$ by the formula
  \begin{equation}
    \label{eqn:01}
    k(s,t)=\int_G\int_G f(v)f(w)k_0(sv,tw)d\mu(v)d\mu(w)
  \end{equation}
  This kernel satisfies the following properties
  \begin{itemize}
  \item[$1)$] $k$ is bounded on every tube.
  \item[$2)$] $k$ is continuous. In fact, $k$ satisfies the following uniform continuity property:
    whenever $s_n\rightarrow s$ and $t_n\rightarrow t$ in $G$, then we have that
    \[\sup_{v\in G} \abs{k(vs_n,vt_n)-k(vs,vt)}\rightarrow 0.\]
  \item[$3)$] if the support of $k_0$ is a tube, then also the support of $k$ is a tube.
  \item[$4)$] if $k_0$ is a positive type kernel, then so is $k$.
  \end{itemize}
\end{lem}
\begin{proof}
  Observe that the new kernel $k$ is well-defined, because for fixed $s,t\in G$, the function
  \[(v,w)\mapsto f(v)f(w)k_0(sv,tw)\]
  is a bounded measurable function with compact support. We check that $k$ satisfies properties $(1)-(4)$.
  
  \textbf{property (1):} Let $T\subseteq G\times G$ be a tube. Observe that $T_0=T(\supp(f)\times\supp(f))$
  is still a tube. So $k_0$ is bounded on $T_0$, say by $C>0$. For any $(s,t)\in T$ and $v,w\in\supp(f)$,
  we get that $(sv,tw)\in T_0$, hence
  \begin{align*}
    \abs{k(s,t)}&\leq \int_G\int_G f(v)f(w)\abs{k_0(sv,tw)} d\mu(v)d\mu(w)\\
    &\leq C\int_G\int_G f(v)f(w) d\mu(v)d\mu(w)=C.
  \end{align*}
  So $k$ is bounded by on $T$.

  \textbf{property (2):} Suppose that $s_n\rightarrow s$ and $t_n\rightarrow t$ in $G$. Let $U$ be a compact neighborhood of identity.
  We can assume that $s_n\in sU$ and $t_n\in tU$ for all $n\in \N$. By assumption, $k_0$ is bounded on the tube
  $T=\Tube(\supp(f)^{-1}U^{-1}s^{-1}tU\supp(f))$, say by $C>0$. Let $r\in G$ be arbitrary. We compute that
  \begin{align*}
    \abs{k(rs_n,rt_n)-k(rs,rt)}
    &\leq\abs{\int_G\int_G f(v)f(w)k_0(rs_nv,rt_nw)-f(v)f(w)k_0(rsv,rtw)d\mu(v)d\mu(w)}\\
    &\leq\int_G\int_G \abs{f(s_n^{-1}v)f(t_n^{-1}w)-f(s^{-1}v)f(t^{-1}w)}\abs{k_0(rv,rw)}d\mu(v)d\mu(w)\\
    &\leq C\int_G\int_G \abs{f(s_n^{-1}v)f(t_n^{-1}w)-f(s^{-1}v)f(t^{-1}w)}d\mu(v)d\mu(w)\\
  \end{align*}
  The last inequality follows from the fact that $(rv,rw)\in T$ whenever either $s_n^{-1}v,t_n^{-1}w\in \supp(f)$ or
  $s^{-1}v,t^{-1}w\in\supp(f)$. The last line does not depend on $r$ and converges to $0$ by the Lebesgue dominated convergence theorem because $f$ is continuous and has compact support.

  \textbf{property (3):} A simple direct computation shows that the support of $k$ is contained in $\supp(k_0)(\supp(f)\times\supp(f))$.

  \textbf{property (4):} It is well-known that a kernel is of positive type if and only if there is a Hilbert space $H$ and a map $\xi:G\rightarrow H$
  such that $k(s,t)=\langle\xi(s),\xi(t)\rangle$ for all $s,t\in G$. The map $\xi$ can be chosen to be weakly Borel if $k$ is Borel.

  Take such a Borel map $\xi^{0}:G\rightarrow H$ for the kernel $k_0$. Observe that $\xi^{0}$ is bounded because $k_0$ is. So the formula
  \[\varphi_s(\eta)=\int_G f(v)\langle\xi^0(sv),\eta\rangle d\mu(v)\]
  defines a bounded anti-linear functional on $H$. By the Riesz representation theorem, there is a unique vector $\xi(s)$ such that
  $\varphi_s(\eta)=\langle\xi(s),\eta\rangle$ for all $\eta\in H$. It now suffices to observe that
  \[k(s,t)=\langle\xi(s),\xi(t)\rangle.\]
\end{proof}

We are now ready to prove theorem \ref{equA}
\begin{proof}[proof of theorem \ref{equA}]

\textbf{1)$\Rightarrow$ 2):}
Let $K\subseteq G$ be a compact subset and let $\epsilon>0$. By lemma \ref{lem:cont}, we only have to find a map $\eta:G\rightarrow L^1(G)$ that satisfies the conditions in (2$^\prime$), i.e,
condition (2), but $\eta$ is not necessarily continuous.

Since $G$ has property $A$, we find a compact subset $L\subseteq G$ and $\{A_s\}_{s\in G}$ a family of Borel subsets in $G\times\N$ with
$0<\mu'(A_s)<\8$ such that
\begin{itemize}
\item for all $(s,t)\in \Tube(K)$ we have
  \begin{align*}
    \frac{\mu'(A_s\Delta A_t)}{\mu'(A_s\cap A_t)}<\frac\epsilon2,
  \end{align*}
\item$(t,n)\in A_s$ implies $(s,t)\in \Tube(L)$.
\end{itemize}

For $s,t\in G$, we denote $A_{t,s}=(\{s\}\times\N)\cap A_t$. It follows from Tonelli's theorem that
\begin{align*}
\int_G|A_{t,s}|d\mu(s)=\int_{G\times \N} \chi_{A_t}(s,n)d\mu'(s,n)=\mu'(A_t)<\8.
\end{align*}
For each $t\in G$, consider the almost everywhere defined measurable map $\eta_t:G\ra \C$ defined by
\begin{align*}
\eta_t(s)= \frac{|A_{t,s}|}{\mu'(A_t)}.
\end{align*}
It is clear that $0\leq \eta_t\in L^1(G)$ and $\|\eta_t\|_1=1$ for all $t\in G$. Note that
\begin{align*}
\|\eta_s\cdot \mu'(A_s)-\eta_t\cdot \mu'(A_t)\|_1&=\int_G \left|\  |A_{s,x}|-|A_{t,x}|\  \right| d\mu(x)\\
%                                               &\leq \int_G  |((\{x\}\times \N)\cap A_s)\Delta((\{x\}\times \N)\cap A_t)| d\mu(x)\\
                                               &\leq\int_G |(\{x\}\times \N) \cap (A_s\Delta A_t)| d\mu(x)\\
                                               &=\mu'(A_s\Delta A_t),
\end{align*}
where the last equality follows from Tonelli's theorem. Hence we see that for all $(s,t)\in\Tube(K)$,
\begin{align*}
||\eta_s-\eta_t||_1&\leq \norm{\eta_s-\eta_t\cdot \frac{\mu'(A_t)}{\mu'(A_s)}}_1+\norm{\eta_t\cdot \frac{\mu'(A_t)}{\mu'(A_s)}-\eta_t}_1\\
                 &\leq \frac{\mu'(A_s\Delta A_t)}{\mu'(A_s)}+\abs{\frac{\mu'(A_t)}{\mu'(A_s)}-1}\\
                 &\leq 2 \cdot \frac{\mu'(A_s\Delta A_t)}{\mu'(A_s)}\\
                 &\leq 2 \cdot \frac{\mu'(A_t\Delta A_s)}{\mu'(A_t\cap A_s)}<\varepsilon.
\end{align*}
Note also that if $\eta_t(s)\neq 0$, then $(s,n)\in A_t$ for some $n$, whence $(t,s)\in \Tube(L)$. Hence $\supp \eta_t\subseteq tL$.

\textbf{2)$\Rightarrow$ 1)}: Given a compact subset $K\subseteq G$ and $\epsilon>0$. We choose a small $0<\epsilon'<1$ such that $\frac{6\epsilon'}{2-5\epsilon'}<\epsilon$. By 2) there exist a compact subset $L\subseteq G$ and a map $\eta:G\ra L^1(G)$ such that $||\eta_t||_1=1$, $\supp \eta_t\subseteq tL$ for every $t\in G$ and
\begin{align*}
\sup_{(s,t)\in \Tube(K)}||\eta_s-\eta_t||_1<\epsilon'.
\end{align*}
We identify $\eta_t$ with a representative function $\eta_t:G\rightarrow\C$.
It is not hard to see that we can assume that $\{s\in G:\eta_t(s)\neq 0\}\subseteq tL$ and $\eta_t$ may also be supposed to be non-negative, since $||\,|\eta_s|-|\eta_t|\,||_1\leq ||\eta_s-\eta_t||_1$.

Note that $\mu(L)>0$, for otherwise $||\eta_t||_1=0$ for all $t\in G$. Let $M:=\mu(L)/\epsilon'>0$. For each $t\in G$, we set
\begin{align*}
A_t:=\{(s,n)\in G\times \N: n\leq \eta_t(s)\cdot M\}.
\end{align*}
It is clear that $A_t$ is a Borel subset of $G\times \N$ for each $t\in G$. For every $t\in G$, we define a measurable map $\theta_t:G\ra [0,\8)$ by 
\begin{align*}
\theta_t(s)=\frac{|A_{t,s}|}{M},
\end{align*}
where $A_{t,s}:=\{n\in \N: (s,n)\in A_t\}$. Then $\theta_t$ satisfies the following two relations
\begin{align*}
\mu'(A_t)&=M\cdot ||\theta_t||_1\\
||\theta_t-\eta_t||_1&<\mu(L)/M=\epsilon',
\end{align*}
for all $t\in G$.
It follows that $M(1-\epsilon')<\mu'(A_t)<M(1+\epsilon')$. In particular, $0<\mu'(A_t)<\8$ for each $t\in G$. Moreover,
\begin{align*}
\mu'(A_t\Delta A_s)=\int_G |A_{t,x}\Delta A_{s,x}|d\mu(x)=\int_G ||A_{t,x}|-|A_{s,x}|| d\mu(x)=M\cdot ||\theta_t-\theta_s||_1.
\end{align*}
Hence,
\begin{align*}
\frac{\mu'(A_s\Delta A_t)}{\mu'(A_s\cap A_t)}=\frac{2\mu'(A_s\Delta A_t)}{\mu'(A_s)+\mu'(A_t)-\mu'(A_s\Delta A_t)}=\frac{2||\theta_s-\theta_t||_1}{||\theta_s||_1+||\theta_t||_1-||\theta_s-\theta_t||_1}.
\end{align*}
Since $||\theta_t||_1>1-\epsilon'$ for every $t\in G$ and $||\theta_s-\theta_t||_1<3\epsilon'$ for all $(s,t)\in \Tube(K)$, we see that
\begin{align*}
\frac{\mu'(A_s\Delta A_t)}{\mu'(A_s\cap A_t)}<\frac{6\epsilon'}{2(1-\epsilon')-3\epsilon'}=\frac{6\epsilon'}{2-5\epsilon'}<\epsilon
\end{align*}
for all $(s,t)\in \Tube(K)$.

Finally, if $(s,n)\in A_t$, then $\eta_t(s)\neq 0$. It follows that $(t,s)\in \Tube(L)$.

\textbf{2)$\Rightarrow$ 3)}: Recall that there is a linear isometric embedding $I: L^1(G)\hookrightarrow C_0(G)^*$ given by
\begin{align*}
I(f)(g)=\int_G gfd\mu.
\end{align*} 
Given a compact subset $K\subseteq G$ and $\epsilon>0$, there exist a compact subset $L\subseteq G$ and a continuous map $\eta:G\ra L^1(G)$ such that $||\eta_t||_1=1$, $\supp\eta_t\subseteq tL$ for every $t\in G$ and
\begin{align*}
\sup_{(s,t)\in \Tube(K)}||\eta_s-\eta_t||_1<\epsilon.
\end{align*}
$\eta_t$ may be supposed to be non-negative (since $|||\eta_s|-|\eta_t|||_1\leq ||\eta_s-\eta_t||_1$). Define the map \hbox{$\nu=I\circ \eta: G\ra C_0(G)^*_+$}, which is obviously a weak-$*$ continuous map with $||\nu_t||=1$. Let $g\in C_0(G)$ be such that $g|_{tL}=0.$ Then
\begin{align*}
\nu_t(g)=I(\eta_t)(g)=\int_{tL} g\eta_t d\mu=0,
\end{align*}
i.e., $\supp \nu_t\subseteq tL$. Finally,
\begin{align*}
\sup_{(s,t)\in \Tube(K)}||\nu_s-\nu_t||=\sup_{(s,t)\in \Tube(K)}||I(\eta_s-\eta_t)||=\sup_{(s,t)\in \Tube(K)}||\eta_s-\eta_t||_1<\epsilon.
\end{align*}

\textbf{3)$\Rightarrow$ 2)}: Let $f$ be a cut-off function for $G$. Using convolution we define a linear contraction $T_f:C_0(G)^*\rightarrow L^1(G)$ by
$T_f(\nu)(s)=\nu(f_s)$
where $f_s$ is defined by $f_s(t)=f(s^{-1}t)$ for $s,t\in G$. Indeed, this is clearly a linear map, and the following computation shows that $T$ is a contraction. Moreover, $T$ is isometric when
restricted to the positive cone $C_0(G)^\ast_+$.
\begin{align*}
||T_f(\nu)||_1=\int_G \abs{\nu(f_s)}d\mu(s)\leq\int_G\int_G f(s^{-1}t)d\abs{\nu}(t)d\mu(s)=\int_G\int_G f(t^{-1}s)d\mu(s)d\abs{\nu}(t)=||\nu||.
\end{align*}
Observe that the inequality above becomes an equality if $\nu$ is positive.
Given a compact subset $K\subseteq G$ and $\epsilon>0$, there exist a compact subset $L\subseteq G$ and a weak-$*$ continuous map $\nu:G\ra C_0(G)^*_+$ such that $||\nu_t||=1$, $\supp \nu_t\subseteq tL$ for every $t\in G$ and
\begin{align*}
\sup_{(s,t)\in \Tube(K)}||\nu_s-\nu_t||<\epsilon.
\end{align*}
Define $\eta:G\ra L^1(G)$ by the composition
\begin{align*}
G \stackrel{\nu} \ra C_0(G)^*_+\stackrel{T_f}\hookrightarrow L^1(G).
\end{align*}
It is clear that $||\eta_t||_1=1$ and $\supp\eta_t\subseteq t(L\cdot \supp f)$ for every $t\in G$. Moreover, we see that
\begin{align*}
\sup_{(s,t)\in \Tube(K)}||\eta_s-\eta_t||_1=\sup_{(s,t)\in \Tube(K)}||T_f(\nu_s-\nu_t)||_1\leq\sup_{(s,t)\in \Tube(K)}||\nu_s-\nu_t||<\epsilon.
\end{align*}

Finally, we show that $\eta$ is continuous. Let $t_n\ra t$, we want to show that $||\nu_{t_n}*f-\nu_t*f||_1\ra 0$. We can assume that $\{t_n\}_{n\in \N}\subseteq t\cdot \text{supp}\ f$. Since $\nu$ is weak-$*$ continuous, we get
\begin{align*}
(\nu_{t_n}*f)(s)=\nu_{t_n}(f_s)\ra \nu_t(f_s)=(\nu_t*f)(s),
\end{align*}
for all $s\in G$. Note that $\text{supp}\ (\nu_{t_n}*f)\subseteq t_n(L\cdot \text{supp}\ f)\subseteq t\cdot \text{supp}\ f\cdot L\cdot \text{supp}\ f$ and
\begin{align*}
(\nu_{t_n}*f)(s)=\nu_{t_n}(f_s)=\int_G f_s(x)d\nu_{t_n}(x)\leq ||f||_\8||\nu_{t_n}||=||f||_\8.
\end{align*}
It follows that 
\begin{align*}
(\nu_{t_n}*f)\leq \chi_{t\cdot \text{supp}\ f\cdot L\cdot \text{supp}\ f} ||f||_\8\in L^1(G)\  a.e. 
\end{align*}
We complete the proof by Lebesgue's dominated convergence theorem.\\

%Finally, we show that $\eta$ is continuous. Let $t_n\ra t$ be a converging sequence in $G$. We want to show that $||\nu_{t_n}*f-\nu_t*f||_1\ra 0$. We can of course assume that $\{t_n\}_{n\in \N}$ is
%contained in the neighborhood $t\cdot \supp\ f$ of $t$. Since $\nu$ is weak-$*$ continuous, we get that
%\begin{align*}
%(\nu_{t_n}*f)(s)=\nu_{t_n}(f_s)\ra \nu_t(f_s)=(\nu_t*f)(s),
%\end{align*}
%for all $s\in G$, so $\nu_{t_n}*f$ converges pointwise to $\nu_t*f(s)$. Note that $\supp (\nu_{t_n}*f)\subseteq t_n(L\cdot \supp f)\subseteq t\cdot \supp f\cdot L\cdot \supp f$ and
%\begin{align*}
%(\nu_{t_n}*f)(s)=\nu_{t_n}(f_s)=\int_G f_s(x)d\nu_{t_n}(x)\leq ||f||_\8||\nu_{t_n}||=||f||_\8.
%\end{align*}
%It follows that the $\nu_{t_n}*f$ are dominated by the integrable function $\chi_{t\cdot \supp f\cdot L\cdot \supp f} ||f||_\8$.
%The Lebesgue's dominated convergence theorem shows that $\nu_{t_n}*f$ converges to $\nu_t*f$ in $||\cdot||_1$.

\textbf{2)$\Rightarrow$ 4)}:
Let $\eta: G\ra L^1(G)$ be a map as in (2). For each $t\in G$, define $\xi_t=|\eta_t|^{1/2}$. Then
\begin{align*}
||\xi_t-\xi_s||_2^2&=\int_{x\in G} |\xi_t(x)-\xi_s(x)|^2 d\mu(x)\\
                             &\leq \int_{x\in G}|\xi_t(x)^2-\xi_s(x)^2| d\mu(x)\\
                             &=\int_{x\in G} ||\eta_t(x)|-|\eta_s(x)|| d\mu(x)\\
                             &\leq ||\eta_t-\eta_s||_1.
\end{align*}
Now, the rest of the proof is obvious.

\textbf{4)$\Rightarrow$ 2)}: Let $\xi:G\ra L^2(G)$ be a map as in (4). For each $t\in G$, define $\eta_t=|\xi_t|^2$. Then by the Cauchy-Schwarz inequality, one has
\begin{align*}
||\eta_t-\eta_s||_1&=\int_{x\in G} ||\xi_t(x)|^2-|\xi_s(x)|^2| d\mu(x)\\
                   &=\int_{x\in G} (|\xi_t(x)|+|\xi_s(x)|)||\xi_t(x)|-|\xi_s(x)|| d\mu(x)\\
                   &\leq |||\xi_t|+|\xi_s|||_2 |||\xi_t|-|\xi_s|||_2\\
                   &\leq 2||\xi_t-\xi_s||_2.
\end{align*}
Now, it is not hard to complete the proof.

\textbf{4)$\Rightarrow$ 5)}:
Given a compact subset $K\subseteq G$ and $\varepsilon>0$, 
let $\xi:G\ra L^2(G)$ be a continuous map as in (4). We identify $\xi_t\in L^2(G)$ with a representative $\xi_t:G\rightarrow \C$. We may assume that $\{s\in G:\xi_t(s)\neq 0\}\subseteq tL$.
Then we define a continuous positive type kernel $k:G\times G\ra \C$ by the formula
\begin{align*}
k(s,t)=\ip{\xi_s}{\xi_t}.
\end{align*}
It is clear that $\supp k\subseteq \Tube(L\cdot L^{-1})$ and we compute that
\begin{align*}
\sup_{(s,t)\in \Tube(K)} |k(s,t)-1|=\sup_{(s,t)\in \Tube(K)}|\ip{\xi_s-\xi_t}{\xi_t}|\leq \sup_{(s,t)\in \Tube(K)} ||\xi_s-\xi_t||_2||\xi_t||_2<\epsilon.
\end{align*}

\textbf{5)$\Rightarrow$ 4)}: Let a compact subset $K\subseteq G$ and $0<\epsilon<1/2$ be given. Let $f$ be a cut-off function for $G$. Observe that $\supp f\cdot (K\cup \{e\})\cdot \supp f$ is a compact subset of $G$. By (5), there exist a compact subset $L\subseteq G$ and a continuous positive type kernel $k_0$ on $G$ such that $\supp k_0\subseteq \Tube(L)$ and
\begin{align*}
\sup\{|k_0(s,t)-1|: (s,t)\in \Tube(\supp f\cdot (K\cup \{e\})\cdot \supp\ f)\}<\epsilon.
\end{align*}
Observe that $k_0$ is bounded because it is of positive type and $k_0(s,s)<1+\varepsilon$ for all $s\in G$.
It follows from lemma \ref{lem:kernel} that $k:G\times G\ra \C$ given by
\begin{align*}
k(s,t)=\int_G\int_G f(v)k_0(sv,tw)f(w)d\mu(w)d\mu(v)
\end{align*}
is a continuous bounded positive type kernel whose support is still a tube, say $\supp(k)\subseteq \Tube(L^\prime)$.
If $(s,t)\in \Tube(K\cup \{e\})$, then
\begin{align*}
|k(s,t)-1|&=\abs{\int_G\int_G f(v)k_0(sv,tw)f(w)d\mu(w)d\mu(v)-\int_G f(v)d\mu(v)\int_G f(w)d\mu(w)}\\
          &\leq \int_{\supp  f}\int_{\supp  f} f(w) f(v)|k_0(sv,tw)-1|d\mu(v)d\mu(w)\\
%          &\leq \sup\{|k_0(x,y)-1|:(x,y)\in \Tube(K\cup \{e\})\cdot (\supp f\times \supp f)\}\\
          &\leq \sup\{|k_0(x,y)-1|:(x,y)\in \Tube(\supp f\cdot (K\cup \{e\})\cdot \supp f)\}\\
          &<\epsilon.
\end{align*}

Let $T_{k_0}$ be the integral operator on $L^2(G)$, which is induced by $k_0$, so we define $T_{k_0}$ by \hbox{$T_{k_0}(\xi)(s)=\int_G k_0(s,t)\xi(t)d\mu(t)$}. Note that $T_{k_0}$ is positive and bounded, and that
\begin{align*}
k(s,t)=\ip{T_{k_0}f_t}{f_s},
\end{align*}
where $f_t(x)=f(t^{-1}x)$.

Let $p$ be a polynomial such that $0\leq p(t)$ and $|p(t)^2-t|<\epsilon/||f||^2_2$ for $t\in[0,||T_{k_0}||]$. Define a continuous map $\eta:G\ra L^2(G)$ by
$\eta_t=p(T_{k_0})f_t$.
Note that
\begin{align*}
|\ip{\eta_t}{\eta_s}-k(s,t)|&=|\ip{p(T_{k_0})f_t}{p(T_{k_0})f_s}-k(s,t)|\\
                            &=|\ip{p^2(T_{k_0})f_t}{f_s}-\ip{T_{k_0}f_t}{f_s}|\\
                            &\leq ||p^2(T_{k_0})-T_{k_0}||||f_t||_2||f_s||_2\\
                            &<\epsilon,
\end{align*}
for all $s,t\in G$. It follows that $|\ip{\eta_t}{\eta_s}-1|\leq |\ip{\eta_t}{\eta_s}-k(s,t)|+|k(s,t)-1|<2\epsilon$ for all \hbox{$(s,t)\in \Tube(K\cup \{e\})$}, which implies that $1-2\epsilon<\text{Re}\ip{\eta_t}{\eta_s}<2\epsilon+1$ for all $(s,t)\in \Tube(K\cup \{e\})$.

Since $\epsilon< 1/2$, we see that $\sqrt{1+2\epsilon}>||\eta_t||_2>\sqrt{1-2\epsilon}>0$. We define a continuous map $\xi:G\ra L^2(G)$ by $\xi_t=\eta_t/||\eta_t||_2$.
This map $\xi$ satisfies
\begin{align*}
1-\text{Re}\ip{\xi_t}{\xi_s}=1-\frac{\text{Re}\ip{\eta_t}{\eta_s}}{\ip{\eta_t}{\eta_t}^{1/2}\ip{\eta_s}{\eta_s}^{1/2}}\leq 1-\frac{1-2\epsilon}{1+2\epsilon}=\frac{4\epsilon}{1+2\epsilon}<4\epsilon
\end{align*}
for all $(s,t)\in \Tube(K\cup \{e\})$. Therefore we see that $||\xi_s-\xi_t||_2=\sqrt{2-2\text{Re}\ip{\xi_t}{\xi_s}}< \sqrt{8\epsilon}$ for all $(s,t)\in \Tube(K)$. Finally, if $p$ is of degree $d$, then it is not hard to see that
\begin{align*}
\supp \xi_t \subseteq t\cdot \supp f\cdot((L^d)^{-1}\cup \cdots \cup L^{-1}\cup \{e\}).
\end{align*} 
\end{proof}

We end this section by showing that property A is equivalent to amenability at infinity. Recall that a locally compact group $G$ is said to be amenable at infinity if there exists a topologically amenable action (in the sense of \cite{CA02}) of $G$ on some compact Hausdorff space $X$. In the discrete case, it is known that $G$ is exact if and only if $G$ is amenable at infinity if and only if the action of $G$ on its Stone-\v{C}ech compactification is topologically amenable if and only if $G$ has property A.
In the locally compact case, we have to replace the Stone-\v{C}ech compactification by the space $\beta^u(G)$
that is defined in the following way. $\beta^u(G)$ is the universal compact Hausdorff left $G$-space equipped with a continuous $G$-equivariant inclusion of $G$ as an open dense subspace, which has the following property: any (continuous) $G$-equivariant map from $G$ into a compact Hausdorff left $G$-space $K$ extends uniquely to a continuous $G$-equivariant map from $\beta^u(G)$ into $K$.
We can identify $C(\beta^u(G))$ with the C$^\ast$-algebra of bounded left-uniform continuous functions on $G$, i.e. the algebra of all bounded continuous functions $f$ on $G$ such that
$f(t^{-1}s)-f(s)$ tends to 0 uniformly as $t$ tends to the unit element of $G$. Anantharaman-Delaroche showed in \cite[Proposition 3.4]{CA02} that a \lcsc group $G$ is amenable at infinity
if and only if its action on $\beta^u(G)$ by translation is topologically amenable.

As in \cite{CA02}, we denote by $\theta$ the homeomorphism of $G\times G$ that is given by $\theta(s,t)=(s^{-1},s^{-1}t)$. Let $C_{b,\theta}(G\times G)$ be the algebra of bounded continuous functions $f$ on $G\times G$ such that $f\circ \theta$ has a continuous extension to $\beta^u(G)\times G$. In fact, we have the following characterization of the continuous functions $f:G\times G\rightarrow \C$ such that $f\circ\theta$ has
a continuous extension to $\beta^u(G)\times G$.
\begin{observation}\label{observation}
Let $f:G\times G\ra \C$ be a (continuous) function. Then $f\circ \theta$ extends to a continuous function on $\beta^u(G)\times G$ if and only if $f$ satisfies the following two conditions:
\begin{align*}
&\sup_{v\in G}|f(v,vt)|<\infty&&\qquad\text{for all }t\in G\\
&\sup_{v\in G}|f(vs_n,vt_n)-f(vs,vt)|\ra 0&&\qquad\text{for all }s_n\ra s\text{ and }t_n\ra t.
\end{align*}
\end{observation}

In \cite{CA02}, Anantharaman-Delaroche showed the following characterization of \lcsc groups that are amenable at infinity.
\begin{theo}[{\cite[Proposition 3.4 and 3.5]{CA02}}]
  Let $G$ be a \lcsc group. Then the following are equivalent.
  \begin{itemize}
    \item[$1)$] $G$ is amenable at infinity, i.e. there exists a topologically amenable action of $G$ on a compact Hausdorff space.
    \item[$2)$] the action of $G$ on $\beta^u(G)$ is topologically amenable.
    \item[$3)$] There exists a net $(k_i)$ of positive type kernels in $C_{b,\theta}(G\times G)$ with support in tubes such that $\lim_i k_i=1$, uniformly on tubes.
  \end{itemize}
\end{theo}

The uniform continuity property in observation \ref{observation} above is precisely the one we obtained in lemma \ref{lem:kernel}.
It is now clear that in point (5) of theorem \ref{equA}, we may assume that the kernel $k$ is contained in $C_{b,\theta}(G\times G)$. This proves the following result
\begin{Corr}
  \label{AAmen}
  A \lcsc group $G$ is
  amenable at infinity if and only if $G$ has property A.
\end{Corr}

One of our motivations to study groups with property A is that for such groups, the Baum-Connes assembly map with coefficients is split-injective.
This was proven first by Higson \cite[Theorem 1.1]{H00} in the discrete case. Later, Chabert, Echterhoff and Oyono-Oyono showed \cite[Theorem 1.9]{CEO04} that this is still true for \lcsc groups
that are amenable at infinity. Since we have just shown that a \lcsc group with property A is amenable at infinity, this is still true for groups with property A.
\begin{Corr}
If $G$ is a locally compact, second countable, Hausdorff group which has property A, then the Baum-Connes assembly map with coefficients for $G$ is split-injective.
\end{Corr}

\section{Uniform embeddability into Hilbert space}
In this section we study groups that admit a uniform embedding into Hilbert space, in the sense of Gromov \cite{G93}, see definition \ref{CA}.
As a consequence of corollary \ref{AAmen} in the previous section, all
groups with property A embed uniformly into Hilbert space. In fact, we show that uniform embeddability into Hilbert space is equivalent to the existence of a Haagerup action on a compact Hausdorff space.
We say that an action $G\actson X$ on a compact Hausdorff space has the Haagerup property if the associated transformation groupoid has a continuous proper conditionally negative type function.
In the previous section, we mentioned that the Baum-Connes assembly map
with coefficients is split-injective for groups with property A. In theorem \ref{thm:BC}, we extend this result to groups that embed uniformly into Hilbert space.

In \cite{G93}, Gromov introduced the notion of a uniform embedding of a metric space into another one. On any \lcsc group $G$, there is a proper left-invariant metric $d$, and this
metric is unique up to coarse equivalence, see \cite{HP06} and \cite{S74}. This gives a well-defined notion of a uniform embedding of a \lcsc group into Hilbert space. However, for the purpose of
this paper, we use the following equivalent definition, that was first given by Anantharaman-Delaroche in \cite{CA02}.

\begin{defi}\label{CA}
Let $G$ be a locally compact, second countable, Hausdorff topological group. A map $u$ from $G$ into a Hilbert space $H$ is said to be a uniform embedding if $u$ satisfies the following two conditions:

a) for every compact subset $K$ of $G$ there exists $R>0$ such that
\begin{align*}
(s,t)\in \Tube(K)\Rightarrow ||u(s)-u(t)||\leq R;
\end{align*}

b) for every $R>0$ there exists a compact subset $K$ of $G$ such that
\begin{align*}
||u(s)-u(t)||\leq R \Rightarrow (s,t)\in \Tube(K).
\end{align*}

  We say that a \lcsc group $G$ embeds uniformly into Hilbert space (or admits a uniform embedding into Hilbert space) if there exists a Hilbert space $H$ and a uniform embedding $u:G\rightarrow H$.
\end{defi}

Anantharaman-Delaroche showed in \cite{CA02} that \lcsc groups that are amenable at infinity, embed uniformly into Hilbert space.
As a consequence of corollary \ref{AAmen}, we obtain the following:
\begin{pro}[\cite{CA02}, Proposition 3.7]
If a \lcsc group $G$ has property A, then $G$ admits a uniform embedding into Hilbert space.
\end{pro}

As with property A, whenever there is a uniform embedding of $G$ into $H$, there also is a continuous uniform embedding of $G$ into $H$.
\begin{pro}
\label{lem:ubcont}
Let $G$ be a locally compact, second countable group. The following are equivalent:\\
1): $G$ admits a uniform embedding into a Hilbert space;\\
2): $G$ admits a Borel uniform embedding into a separable Hilbert space;\\
3): $G$ admits a continuous uniform embedding into a separable Hilbert space.
\end{pro}
\begin{proof} It is clear that 3) implies 1).

\textbf{1)$\Rightarrow$ 2):} Let $u:G\ra H$ be a uniform embedding into a Hilbert space $H$. Let $C$ be a compact neighborhood of identity in $G$. As in the proof of
lemma \ref{lem:cont}, we find group elements $(s_n)_n$ and Borel subsets $C_n\subseteq s_nC$ such that $G=\bigsqcup_{n}C_n$. Define $u^\prime:G\ra H$ by the property that $u^\prime(s)=u(s_n)$ whenever $s\in C_n$.
This way, $u^\prime$ is a Borel step function. Since $u$ is a uniform embedding, we find an $R>0$ such that $\norm{u(s)-u(t)}\leq R$ whenever $s^{-1}t\in C$. Fix $n\in\N$ and observe that every $s\in C_n$ satisfies $s_n^{-1}s\in C$. As a consequence,
\[\norm{u^\prime(s)-u(s)}=\norm{u(s_n)-u(s)}\leq R.\]
Since this is true for all $n\in \N$ and $s\in C_n$, we see that $u^\prime$ is at bounded distance from $u$, so $u^\prime$ is still a uniform embedding.
Observe that $u^\prime$ takes values only in the separable closed subspace $H_0\subseteq H$ that is spanned by $\{u(s_n): n\in \N\}$. In other words, $u^\prime$ is
a Borel uniform embedding into the separable Hilbert space $H_0$.

\textbf{2)$\Rightarrow$ 3):} Let $u:G\ra H$ be a Borel uniform embedding into a separable Hilbert space $H$. Let $f$ be a cut-off function for $G$. Since $u$ is a uniform embedding, we find $R>0$ such that
$\norm{u(s)-u(t)}\leq R$ whenever $s^{-1}t\in \supp(f)$.
For a fixed $t\in G$, we define an anti-linear functional $\varphi_t:H\ra \C$ by the formula $\varphi_t(v)=\int_G\ip{f(s^{-1}t)u(s)}{v}d\mu(s)$ for all $v\in H$. Observe that $\varphi_t$ is bounded because
\[\abs{\varphi_t(v)}\leq \int_G\abs{\ip{f(s^{-1}t)u(s)}{v}}d\mu(s)\leq \int_Gf(s^{-1}t)(\norm{u(t)}+R)\norm{v}d\mu(s)=(\norm{u(t)}+R)\norm{v},\]
for every vector $v\in H$.
By the Riesz-Fr\'{e}chet theorem there exists a unique $u^\prime(t)\in H$ such that $\varphi_t(v)=\ip{u^\prime(t)}{v}$ for all $v\in H$. Observe that $u^\prime(t)$ is at distance at most $R$ from $u(t)$:
\begin{align*}
  \abs{\ip{u^\prime(t)-u(t)}{v}}&= \abs{\int_G f(s^{-1}t)\ip{u(s)}{v}d\mu(s)-\ip{u(t)}{v}\int_G f(s^{-1}t)d\mu(s)}\\
  &\leq \int_G f(s^{-1}t)R\norm{v}d\mu(s)\\
  &\leq R\norm{v}.
\end{align*}
In particular, we see that $u^\prime$ is still a uniform embedding of $G$ into $H$.

We show that $u^\prime$ is continuous. This follows from the following computation: let $(t_n)_n$ be a sequence in $G$ that converges to $t\in G$. The sequence $t^{-1}t_n$ remains in some compact
neighborhood $U$ of identity in $G$. Since $u$ is a uniform embedding, we find $R_2>0$ such that $\norm{u(s)-u(t)}\leq R_2$ whenever $t^{-1}s\in U\supp(f)$.
Now we see that, for every $v\in H$ and $n\in \N$,
\begin{align*}
  \abs{\ip{u^\prime(t_n)-u^\prime(t)}{v}}&\leq\int_G \abs{f(s^{-1}t_n)-f(s^{-1}t)}\norm{u(s)}\norm{v}d\mu(s)\\
  &=\int_G\abs{f(t_n^{-1}s)-f(t^{-1}s)}\norm{u(s)}\norm{v} d\mu(s)\\
  &\leq \norm{t_n\cdot f-t\cdot f}_\infty\mu(U\supp(f))(\norm{u(t)}+R_2)\norm{v}.
\end{align*}
Since $f$ is continuous with compact support, we see that $\norm{t_n\cdot f- t\cdot f}_\infty$ tends to 0. Therefore we also get that $\norm{u^\prime(t_n)-u^\prime(t)}$ tends to 0.
\end{proof}

We give an alternative characterization of uniform embedding into Hilbert space in terms of transformation groupoids and conditionally negative type functions on it. For the convenience of our readers, we recall these concepts:

Let $G$ be a locally compact group acting continuously on a locally compact Hausdorff space $X$. The transformation groupoid $X\rtimes G$ consists of all pairs $(x,g)$ with $x\in X$, $g\in G$. Its base space is $X$, and the source and range maps are given by
\begin{align*}
s(x,g)=g^{-1}x, \quad r(x,g)=x.
\end{align*}
The composition law is $(gx,g)(x,g')=(gx,gg')$ and the inversion is given by $(x,g)^{-1}=(g^{-1}x,g^{-1})$.

A conditionally negative type function on $X\rtimes G$ is a function $\psi: X\times G\ra \R$ such that\\
1) $\psi(x,e)=0$ for all $x\in X$;\\
2) $\psi(x,g)=\psi(g^{-1}x,g^{-1})$ for all $(x,g)\in X\times G$;\\
3) $\sum_{i,j=1}^nt_it_j \psi(g_i^{-1}x,g_i^{-1}g_j)\leq 0$ for all $\{t_i\}_{i=1}^n \subseteq \mathbb{R}$ satisfying $\sum_{i=1}^n t_i=0$, $g_i\in G$ and $x\in X$.

We say that an action $G\actson X$ of a group on a compact Hausdorff space has the Haagerup property if its transformation groupoid $X\rtimes G$ admits a continuous proper conditionally negative type function.

\begin{theo}\label{u.e.Groupoid}
Let $G$ be a \lcsc group. The following are equivalent:\\
1): $G$ admits a uniform embedding into a Hilbert space.\\
2): There exists a continuous conditionally negative type kernel $k$ on $G\times G$ satisfying
\begin{itemize}
\item{$k$ is bounded on every tube;}
\item{$k$ is a proper kernel, i.e. $\{(s,t)\in G\times G: |k(s,t)|\leq R\}$ is a tube for all $R>0$.}
\end{itemize}
3): The action $G\actson\beta^u(G)$ has the Haagerup property, i.e, there exists a continuous proper conditionally negative type function on $\beta^u(G)\rtimes G$.\\
4): There exists a second countable compact Hausdorff left $G$-space $Y$ which admits a continuous proper conditionally negative type function on $Y\rtimes G$.
\end{theo}
\begin{proof}
%\textbf{1)$\Rightarrow$ 2):} It follows from proposition \ref{lem:ubcont} that $G$ admits a continuous uniform embedding $u:G\ra H$, where $H$ is a Hilbert space. Define a continuous conditionally negative type kernel $k:G\times G\ra \R$ by
%\begin{align*}
%k(s,t)=||u(s)-u(t)||^2.
%\end{align*}
%Since $u$ is a uniform embedding, it is clear that $k$ satisfies the required conditions in 2).
%
\textbf{2)$\Rightarrow$ 1):} Assume that $k$ is a continuous conditionally negative type kernel on $G\times G$ satisfying the conditions in 2). It follows from the GNS construction (see Theorem C.2.3 \cite{BHV08}) that there exist a real Hilbert space $H$ and a continuous map $u:G\ra H$ such that
\begin{align*}
k(s,t)=||u(s)-u(t)||^2.
\end{align*}
By the conditions on $k$, it is easy to see that $u$ is a uniform embedding.

\textbf{1)$\Rightarrow$ 3):} We may assume that $G$ admits a continuous uniform embedding $u:G\ra H$, where $H$ is a (separable) Hilbert space. Define a continuous function $k_0:G\times G\ra \R$ by
\begin{align*}
k_0(s,t)=||u(s)-u(t)||^2.
\end{align*}
It is clear that $k_0$ is bounded on every tube. Let $f$ be a cut-off function for $G$. It follows from lemma \ref{lem:kernel} that the kernel $k:G\times G\ra \R$ that is given by
\begin{align*}
k(s,t)=\int_G\int_G f(v)k_0(sv,tw)f(w)d\mu(w)d\mu(v)
\end{align*}
is continuous and bounded on every tube. Moreover, $k\circ \theta$ has a continuous extension $\psi_0:\beta^u(G)\times G\ra \R$.

We have already seen in the proof of the previous proposition that there exists a unique continuous uniform embedding $u*f:G\ra H$ such that
\begin{align*}
\ip{u*f(t)}{\eta}=\int_G \ip{f(s^{-1}t)u(s)}{\eta}d\mu(s)\quad \text{for $\eta\in H$}.
\end{align*}
Now, we define a continuous conditionally negative type kernel $\phi:G\times G\ra \R$ by
\begin{align*}
\phi(s,t)=||u*f(s)-u*f(t)||^2.
\end{align*}
By the definition of $u*f$, it is not hard to see that
\begin{align*}
\phi(s,t)&=\text{Re}\int_G\int_G f(v)\ip{u(sv)-u(tv)}{u(sw)-u(tw)}f(w)d\mu(w)d\mu(v)\\
         &=\int_G\int_G f(v)\text{Re}{\ip{u(sv)-u(tv)}{u(sw)-u(tw)}}f(w)d\mu(w)d\mu(v)\\
         &=\frac12\int_G\int_G f(v)f(w)
         \left(\begin{aligned}
             &-\norm{u(sw)}^2 + 2\text{Re}{\ip{u(sv)}{u(sw)}} -\norm{u(sv)}^2\\
             &+ \norm{u(sv)}^2 - 2\text{Re}{\ip{u(sv)}{u(tw)}} + \norm{u(tw)}^2\\
             &- \norm{u(tw)}^2 +2\text{Re}\ip{u(tv)}{u(tw)} - \norm{u(tv)}^2\\
             &+ \norm{u(tv)}^2 - 2\text{Re}\ip{u(tv)}{u(sw)} + \norm{u(sw)}^2\end{aligned}\right)
         d\mu(w)d\mu(v)\\
         &=\frac12\int_G\int_G f(v)f(w)(-k_0(sv,sw)+k_0(sv,tw)-k_0(tv,tw)+k_0(tv,sw))d\mu(w)d\mu(v)\\
         &=\frac{1}{2}(k(s,t)-k(s,s)-k(t,t)+k(t,s))\\
         &=\frac{1}{2}(2k(s,t)-k(s,s)-k(t,t)),
\end{align*}
where the first equality follows from the fact that $\phi$ is real-valued.

Thus, the function $\psi:\beta^u(G)\times G\rightarrow\R$ that is given by $\psi(y,t)=\psi_0(y,t)-\frac12(\psi_0(y,e)+\psi_0(t^{-1}y,e))$ extends $\phi\circ \theta$ continuously.
Note that $\phi$ is a conditionally negative type kernel on $G\times G$ if and only if $\phi\circ \theta$ is a conditionally negative type function on $G\rtimes G$, which is also equivalent to $\psi$ being a conditionally negative type function on $\beta^u(G)\rtimes G$. Moreover, $\psi$ is proper because $\{(s,t)\in G\times G: |\phi(s,t)|\leq R\}$ is a tube for all $R>0$.

\textbf{3)$\Rightarrow$ 4):} Let $\phi:\beta^u(G)\times G\ra \R$ be a continuous proper conditionally negative type function on $\beta^u(G)\rtimes G$. If we identify $C(\beta^u(G)\times G)$ with $C(G,C(\beta^u(G)))$, then $G\ni t\mapsto \phi(\cdot,t)\in C(\beta^u(G))$ is a continuous map. Let $A$ be the $C^*$-algebra generated by the unit in $C(\beta^u(G))$ and the set $\{s.\phi(\cdot,t):s,t\in G\}$. It is clear that $A$ is a unital, separable and $G$-invariant $C^*$-subalgebra of $C(\beta^u(G))$. Hence, there exists a compact Hausdorff, second countable left $G$-space $Y$ such that $A\cong C(Y)$. It is not hard to see that there exist a continuous $G$-equivariant surjection $p:\beta^u(G)\ra Y$ and a continuous function $\psi:Y\times G\ra \R$ such that the following diagram
$$
\xymatrix{
\beta^u(G)\times G \ar[dr]^{\phi}\ar[d]_{p\times \text{id}}\\
Y\times G \ar[r]^{\psi} &\R&
}
$$ 
commutes. The properness of $\psi$ follows from the properness of $\phi$ and the surjectivity of $p$. Since $p$ is also $G$-equivariant, $\psi$ is a conditionally negative type function on $Y\rtimes G$, as desired.

\textbf{4)$\Rightarrow$ 2):} Let $\varphi:Y\times G\rightarrow \R$ be a conditionally negative type function on $Y\rtimes G$. Fix one point $y_0\in Y$ and define a kernel $k:G\times G\rightarrow \R$
by the following formula:
\[k(s,t)=\varphi(s^{-1}y_0,s^{-1}t)\quad\text{for all }s,t\in G.\]
It is now clear that $k$ is a continuous function. Because $\varphi$ was a conditionally negative function, one easily computes that $k$ is a conditionally negative type kernel on $G$. Moreover,
because $Y$ is compact and $\varphi$ is continuous, it follows that $k$ is bounded on tubes. Finally, the properness of $\varphi$ translates to the properness of $k$, as a kernel on $G$.
\end{proof}

Our final result shows that for every group $G$ that embeds uniformly into a Hilbert space, the Baum-Connes assembly map with coefficients is split-injective. The analogous result for discrete groups
was first proven by Skandalis, Tu and Yu (\cite{STY02} Theorem 6.1). The argument is almost identical to the one used to prove \cite[Theorem 1.9]{CEO04}.
\begin{theo}\label{thm:BC}
If $G$ is a \lcsc group which admits a uniform embedding into Hilbert space, then the Baum-Connes assembly map
\begin{align*}
\mu_A:K_*^\emph{top}(G;A)\ra K_*(A\rtimes_r G)
\end{align*} 
is split-injective for any separable $G$-$C^*$-algebra $A$.
\end{theo}
\begin{proof}
Suppose that $\psi$ is a continuous proper conditionally negative type function on $Y\rtimes G$ as in theorem \ref{u.e.Groupoid} 4). We show first that we can assume that $Y$ is a compact convex space,
on which $G$ acts by affine transformations. Let $X$ denote the space Prob($Y$) of Borel probability measures on $Y$ equipped with the weak-$*$ topology. Notice that $X$ is a second countable compact Hausdorff left $G$-space (with the induced action from $G\curvearrowright Y$). We define $\phi:X\times G\ra \R$ by
\begin{align*}
\phi(m,t)=\int_Y\psi(y,t)dm(y).
\end{align*}
We claim that $\phi$ is a continuous proper conditionally negative type function on $X\rtimes G$. Indeed, if we identify $X$ with the state space of $C(Y)$, then we see that $\phi(m,t)=m(\psi(\cdot,t))$. Thus, the continuity of $\phi$ follows from the norm-boundedness of $X$ and the continuity of the map $G\ni t\mapsto \psi(\cdot,t)\in C(Y,\R)$. It is not hard to see that $\phi$ is a continuous proper conditionally negative type function since $\psi$ is.

We now consider the following commutative diagram, which is called the Higson descent diagram (cf. the proof of theorem 3.2 in \cite{H00}):
$$
\xymatrix{\ar @{} [ddrr] |{}
K_*^\text{top}(G;A)\ar@{->}[dd]_{i_*} \ar[rr]^{\mu_A} &&K_*(A\rtimes_r G) \ar@{->}[dd]^{i_*} \\ \\
K_*^\text{top}(G;A\otimes C(X)) \ar[rr]_{\mu_{A\otimes C(X)}} && K_*((A\otimes C(X))\rtimes_r G),} \ \ 
$$
where the vertical arrows are induced by the inclusion $i:\C\ra C(X)$ and the horizontal arrows are the Baum-Connes assembly maps. By lemma 4.1 in \cite{STY02}, the Baum-Connes assembly map for the groupoid $X\rtimes G$ with coefficients in $A\otimes C(X)$ is the same as the one for the group $G$ with coefficients in $A\otimes C(X)$. Because whenever the groupoid $X\rtimes G$ has a continuous proper conditionally
negative type function, it also has a proper affine isometric action on a continuous field of Hilbert spaces over $X$ \cite{Tu99}.
Hence, by theorem 9.3 in \cite{Tu99} the bottom horizontal arrow is an isomorphism. Since $X$ is convex and the action of $G$ on $X$ is affine, the space $X$ is $K$-equivariantly contractible for any compact subgroup $K$ of $G$. By proposition 1.10 in \cite{CEO04}, the left vertical arrow is an isomorphism. An easy diagram chase then shows the split-injectivity of the assembly map $\mu_A$.
\end{proof}
%%%%%%%%%%% End Appendix %%%%%%%%%%%

%%%%%%%%%% Begin of Bibliography %%%%%%%%%%
\markboth{References}{}
\addtocontents{toc}{\vspace{5mm}\hrule}
\addcontentsline{toc}{section}{\numberline {}References}
\addtocontents{toc}{\vspace{5mm}\hrule}
\bibliographystyle{plain}
\small
\bibliography{bibliography}

\quad Kang Li, Department of Mathematical Sciences, University of Copenhagen, \\
Universitetsparken 5, DK-2100 Copenhagen, Denmark
    
{\it E-mail address:}  kang.li@math.ku.dk

\quad Steven Deprez, Department of Mathematical Sciences, University of Copenhagen, \\
Universitetsparken 5, DK-2100 Copenhagen, Denmark
    
{\it E-mail address:}  sdeprez@math.ku.dk
%%%%%%%%%%% End of Bibliography %%%%%%%%%%%
\end{document}